\documentclass[reqno]{amsart}%
\usepackage{amstext}
\usepackage{amsfonts}
\usepackage{amsmath}
\usepackage{amssymb}
\usepackage{hyperref}
\usepackage{graphicx}%
\providecommand{\U}[1]{\protect\rule{.1in}{.1in}}
\numberwithin{equation}{section}
\newtheorem{theorem}{Theorem}[section]

\newtheorem{proposition}[theorem]{Proposition}

\newtheorem{lemma}[theorem]{Lemma}

\begin{document}
\title[Navier-Stokes equations ]{Uniqueness criterion of weak solutions for the 3D Navier-Stokes equations}
\author{Abdelhafid Younsi}
\address{Department of Mathematics and Computer Science, University of Djelfa , Algeria.}
\email{younsihafid@gmail.com}

\begin{abstract}
In this paper we establish a new uniqueness result of weak solutions for the
3D Navier-Stokes equations. Under assumption that there is not uniqueness of
weak solution in singular time, we prove that if two weak solutions $u$ and
$v$ of 3D Navier-Stokes equations belong to $L^{\frac{5}{2}}\left(
0,T;V\right)  $ with the same initial datum, then we get $u=v$. In the class
$L^{3}\left(  0,T;V\right)  $, we prove that $u-v\in$ $C^{0}\left(
0,T;H\right)  $ and $u=v\ $when $u_{0}=v_{0}$.

\end{abstract}
\maketitle

\section{Introduction}

For the three dimensional Navier-Stokes system weak solutions are known to
exist by a basic result due to J. Leray from 1934 \cite{5}. It is not known if
a weak solution is unique or what further assumption could make it unique.
Therefore, the uniqueness of weak solutions remains as an open problem.

There are many results that give sufficient conditions for regularity of weak
solutions \cite{4, 6, 7} and \cite{10}. The Prodi-Serrin condition (see Serrin
\cite{8}, Prodi \cite{6}) states that any weak Leray-Hopf solution verifying
$u\in L^{p}(0,\infty;L^{q}\left(
%TCIMACRO{\U{211d} }%
%BeginExpansion
\mathbb{R}
%EndExpansion
^{3}\right)  $ with $2/p+3/q=1$, $2\leq p<\infty$, is regular on $\left(
0,\infty\right)  \times%
%TCIMACRO{\U{211d} }%
%BeginExpansion
\mathbb{R}
%EndExpansion
^{3}$. In 1995, Beirao \cite{1} obtained uniqueness results in the class
$L^{4}(0,T;V)$.

It was Leray's conjecture, that the weak solutions of 3D Navier-Stokes
equations develop singularities (\cite{7, 9, 10}) and it is not yet proved or
disproved. For any weak solution the set of singular times $\Sigma=\left\{
t\in\left[  0,T\right]  ;\Vert u\left(  t\right)  \Vert_{H^{1}}=\infty
\right\}  $, Leray \cite{5} has proved that Lebesgue measure of $\Sigma$ is
zero. Caffarelli et al \cite{4} proved that $\Sigma$ has one-dimensional
Hausdorff measure zero; see also Sohr \cite{9} and Scheffer \cite{7}.

In this paper, we are interested in the uniqueness of Leray-Hopf weak
solutions to the 3D Navier-Stokes equations. In our main result, we prove that
if there is not uniqueness when $t\in\Sigma$, then the Leray-Hopf solutions
are unique in the class $L^{\frac{5}{2}}(0,T;V)$. Under the same assumption,
our second result states if two weak solutions $u$ and $v$ are in
$L^{3}\left(  0,T;V\right)  $, then their difference $w=u-v$ is strong
continuous from $\left[  0,T\right]  $ into $H$ and $u=v$ when $u_{0}=v_{0}$.

\section{Uniqueness criterion for weak solutions}

In this paper, we consider the three-dimensional Navier-Stokes system
\begin{equation}%
\begin{array}
[c]{c}%
\dfrac{\partial u}{\partial t}+u.\nabla u=\nu\triangle u+\nabla p+f,\text{
}t>0,\\
\text{div }u=0\text{, in }\Omega\times\left(  0,\infty\right)  \text{,
}u=0\text{ on }\partial\Omega\times\left(  0,\infty\right)  \text{ and
}u\left(  x,0\right)  =u_{0}\text{, in }\Omega\text{,}%
\end{array}
\label{1}%
\end{equation}
where $u=u\left(  x,t\right)  \ $is the velocity vector field, $p\left(
x,t\right)  $ is a scalar pressure, $f$ is a given force field and $\nu>$ $0$
is the viscosity of the fluid. $u\left(  x,0\right)  $ with div$u_{0}=0$ in
the sense of distribution is the initial velocity field, and $\Omega$ is a
regular, open, bounded subset of $%
%TCIMACRO{\U{211d} }%
%BeginExpansion
\mathbb{R}
%EndExpansion
^{3}$ with smooth boundary $\partial\Omega$. We denote by $H^{1}\left(
\Omega\right)  $, the Sobolev space. We define the usual function spaces
$\mathcal{V}=\left\{  u\in C_{0}^{\infty}\left(  \Omega\right)  \text{: div
}u=0\right\}  $, $V=$ closure of $\mathcal{V}$ in $H_{0}^{1}\left(
\Omega\right)  $ and $H=$ closure of $\mathcal{V}$ in $L^{2}(\Omega)$. The
space $H$ is equipped with the scalar product $(.,.)$ induced by $L^{2}%
(\Omega)$ and the norm $\left\Vert .\right\Vert _{L^{2}}$ and the space $V$ is
equipped with the norm $\left\Vert .\right\Vert _{H^{1}}$. We denote by
$H^{\prime}$ and $V^{\prime}$ the dual spaces of $H$ and $V$.

We recall that a Leray weak solution of the Navier-Stokes equations is a
bounded solution and weakly continuous from $\left[  0,T\right]  $ into $H$,
whose gradient is square-integrable in space and time and that satisfies the
energy inequality. The proof of the following theorem is given in
\cite[Theorem 3.1 P21]{10} and \cite[Theorem (Leray) P71]{2}.

\begin{theorem}
Let $\Omega\subset%
%TCIMACRO{\U{211d} }%
%BeginExpansion
\mathbb{R}
%EndExpansion
^{3}$, $f\in L^{2}(0,T;V^{\prime})$ and $u_{0}\in V_{0}$ be given. Then there
exists a weak solution $u$ of the 3D Navier-Stokes equations $(\ref{1}%
)$\ which satisfies $u\in L^{2}(0,T;V)\cap L^{\infty}(0,T;H),\forall T>0.$
\end{theorem}

Let $w=u-v$ the difference of two solutions $u$ and $v$ satisfying the same
initial and Theorem 2.1. We introduce the following quantity%
\begin{equation}
\Lambda=\rho_{1}\Vert w\left(  .,t\right)  \Vert_{L^{2}}^{\frac{3}{2}}%
+\frac{\rho_{2}}{1+\Vert u\left(  .,t\right)  \Vert_{H^{1}}}, \label{2}%
\end{equation}
with $\rho_{1}$ and $\rho_{2}$ two positive constants. If the weak solutions
of the Navier-Stokes equations develop singularities, as Leray has
conjectured, see \cite{2}, \cite{10} and references therein, it follows that
$\Vert u\left(  t\right)  \Vert_{H^{1}}$ is finite for all $t\notin\Sigma$.
Thus, $\Lambda$ is finite and different to zero when $t\notin\Sigma$ and then
$1/\Lambda\ $is bounded. For $t\in\Sigma$ yields
\begin{equation}
\Lambda\rightarrow\rho_{1}\left\Vert w\left(  t\right)  \right\Vert _{L^{2}%
}^{\frac{3}{2}}\text{ and }\Vert u\left(  t\right)  \Vert_{H^{1}}=\infty.
\label{2a}%
\end{equation}
Therefore, $1/\Lambda\ $is unbounded if
\begin{equation}
\left\Vert w\left(  t\right)  \right\Vert _{L^{2}}=0\text{ when }\Vert
u\left(  t\right)  \Vert_{H^{1}}=\infty, \label{2b}%
\end{equation}
this case is an open problem in the theory of Navier-Stokes equations. For
related question, we refer to Constantin \cite[Chapter Ten]{2}
\textquotedblleft\ Loss of regularity in blow up\textquotedblright. In order
to overcome this difficulty, we assume in what follows that
\begin{equation}
\left\Vert w\left(  t\right)  \right\Vert _{L^{2}}\neq0\text{ when }\Vert
u\left(  t\right)  \Vert_{H^{1}}=\infty, \label{2c}%
\end{equation}
This assumption assures that $\Lambda\neq0$ for all $t\geq0$. According to
this hypothesis we show the following uniqueness theorem.

\begin{theorem}
Let $u$ a weak solution of the 3D Navier-Stokes equations satisfying
$(\ref{1})$. If $u\in L^{\frac{5}{2}}(0,T;V)$, then such a solution $u$ is unique.
\end{theorem}

\begin{proof}
Let $w=$ $u-v$ the difference of two weak solutions $u$ and $v$ for the 3D
Navier-Stokes equations $(\ref{1})$ satisfying the same initial. Recall that
$w$ satisfies the following differential inequality%
\begin{equation}
\frac{1}{2}\frac{d}{dt}\left\Vert w\right\Vert _{L^{2}}^{2}+\nu\left\Vert
w\right\Vert _{H^{1}}^{2}\leq c\Vert w\Vert_{H^{1}}^{\frac{3}{2}}\Vert
w\Vert_{L^{2}}^{\frac{1}{2}}\Vert u\Vert_{H^{1}}, \label{3}%
\end{equation}
$c$ is a positive constant, see \cite{1} and \cite{9, 10}.\ Under assumption
that $\left\Vert w\right\Vert _{L^{2}}\neq0$ for all $t$ $\in\Sigma$, it
follows that $\Lambda$ is finite and different to zero\ for all $t\geq0$ and
thus $\frac{1}{\Lambda}$ is finite for all $t\geq0$. Consequently there exists
a positive constant $\mu$ such that
\begin{equation}
\frac{1}{\Lambda}=\mu<\infty,\forall\text{\ }t\geq0. \label{4}%
\end{equation}
From the differential inequality $(\ref{3})$ we get for all $t\geq0$%
\begin{equation}
\frac{1}{2}\frac{d}{dt}\left\Vert w\right\Vert _{L^{2}}^{2}+\nu\left\Vert
w\right\Vert _{H^{1}}^{2}\leq c\frac{\Lambda}{\Lambda}\Vert w\Vert_{H^{1}%
}^{\frac{3}{2}}\Vert w\Vert_{L^{2}}^{\frac{1}{2}}\Vert u\Vert_{H^{1}}.
\label{5}%
\end{equation}
Using $\mu$ in $(\ref{5})$, it follows that
\begin{equation}
\frac{1}{2}\frac{d}{dt}\left\Vert w\right\Vert _{L^{2}}^{2}+\nu\left\Vert
w\right\Vert _{H^{1}}^{2}\leq c\rho_{1}\mu\Vert w\Vert_{L^{2}}^{2}\Vert
w\Vert_{H^{1}}^{\frac{3}{2}}\Vert u\Vert_{H^{1}}+c\rho_{2}\mu\Vert
w\Vert_{H^{1}}^{\frac{3}{2}}\Vert w\Vert_{L^{2}}^{\frac{1}{2}}, \label{6}%
\end{equation}
where we have used that $\frac{\Vert u\Vert_{H^{1}}}{1+\Vert u\Vert_{H^{1}}%
}\leq1$ for all $t\geq0$. Using the Poincar\'{e} inequality in the second term
on the right-hand side of $(\ref{6})$, and by taking $\rho_{1}=\dfrac{1}{\mu}$
and $\rho_{2}=\dfrac{\nu\lambda_{1}^{\frac{1}{4}}}{2c\mu}$, the following
estimate holds%
\begin{equation}
\frac{1}{2}\frac{d}{dt}\left\Vert w\right\Vert _{L^{2}}^{2}+\nu\left\Vert
w\right\Vert _{H^{1}}^{2}\leq c\Vert w\Vert_{L^{2}\left(  \Omega\right)  }%
^{2}\Vert w\Vert_{H^{1}}^{\frac{3}{2}}\Vert u\Vert_{H^{1}}+\frac{\nu}{2}\Vert
w\Vert_{H^{1}}^{2}. \label{7}%
\end{equation}
Thus, the right-hand side of $(\ref{7})$ is independent of $\mu$. Finally, we
obtain%
\begin{equation}
\frac{d}{dt}\left\Vert w\right\Vert _{L^{2}}^{2}\leq c\Vert w\Vert_{L^{2}}%
^{2}\Vert w\Vert_{H^{1}}^{\frac{3}{2}}\Vert u\Vert_{H^{1}}. \label{8}%
\end{equation}
Applying Gronwall's inequality on $(\ref{8})$ in the interval $\left[
0,T\right]  $ yields
\begin{equation}
\left\Vert w\left(  t\right)  \right\Vert _{L^{2}}^{2}\leq c\Vert w\left(
0\right)  \Vert_{L^{2}}^{2}\exp\left(  \int_{0}^{T}\Vert u\Vert_{H^{1}}\Vert
w\Vert_{H^{1}}^{\frac{3}{2}}dt\right)  . \label{9}%
\end{equation}
Applying Holder's inequality on $\int_{0}^{T}\Vert u\Vert_{H^{1}}\Vert
w\Vert_{H^{1}}^{\frac{3}{2}}dt$ with $p=\frac{5}{2}$ and $q=\frac{5}{3}$, we
find%
\begin{equation}
\int_{0}^{T}\Vert u\Vert_{H^{1}}\Vert w\Vert_{H^{1}}^{\frac{3}{2}}%
dt\leq\left(  \int_{0}^{T}\Vert u\Vert_{H^{1}}^{\frac{5}{2}}dt\right)
^{\frac{2}{5}}+\left(  \int_{0}^{T}\Vert w\Vert_{H^{1}}^{\frac{5}{2}%
}dt\right)  ^{\frac{3}{5}}. \label{10}%
\end{equation}
Therefore, since $u\in L^{\frac{5}{2}}(0,T;V)$ and\ $\Vert w\left(  0\right)
\Vert_{L^{2}}=0$, it follows from $(\ref{9})$ and $(\ref{10})$ that
$\left\Vert w\left(  t\right)  \right\Vert _{L^{2}}=0$ for all $t\leq T$. The
theorem is thus proved.
\end{proof}

Since we only have $\partial_{t}u\in L^{\frac{4}{3}}\left(  0,T;V^{\prime
}\right)  $ and $u\in L^{2}\left(  0,T;V\right)  $ the continuity of the weak
solutions of the 3D Navier-Stokes equations is known to be proved only in this
weak sense, see \cite{10} and the references therein. The purpose of the
following lemma is to establish the continuity of $\left\Vert w\right\Vert
_{L^{2}}$ and gives enough regularity for $w$ to deduce that $\left(
\partial_{t}w,w\right)  =\frac{d}{dt}\left\Vert w\right\Vert ^{2}$.

\begin{lemma}
If $u$ and $v$ are two weak solutions of the 3D Navier-Stokes equations
$(\ref{1})$ with $u$, $v\in L^{3}\left(  0,T;V\right)  $, then $w=u-v$ is a
continuous function $\left[  0,T\right]  \rightarrow H$.
\end{lemma}

\begin{proof}
We consider two solutions $u$ and $v$ of $(\ref{1})$, and write the equation
for their difference $w=u-v$. Then $w$ satisfies%
\begin{equation}%
\begin{array}
[c]{ll}%
\partial_{t}w & =-\triangle w-B\left(  u,w\right)  -B\left(  w,v\right) \\
& =-\triangle w-B\left(  v,w\right)  -B\left(  w,u\right)  .
\end{array}
\label{11}%
\end{equation}
Clearly $\triangle u\in L^{2}\left(  0,T;V^{\prime}\right)  $, since $u\in
L^{2}\left(  0,T;V\right)  $. So we consider, for $\phi\in V_{1}$ with
$\left\Vert \nabla\phi\right\Vert _{L^{2}}=1$, we obtain the following
estimate%
\begin{equation}%
\begin{array}
[c]{ll}%
\left\vert \left(  B\left(  v,w\right)  ,\phi\right)  -\left(  B\left(
w,u\right)  ,\phi\right)  \right\vert  & \leq\left\vert \left(  B\left(
v,\phi\right)  ,w\right)  -\left(  B\left(  w,\phi\right)  ,u\right)
\right\vert \\
& \leq\left(  \left\Vert u\right\Vert _{L^{4}}+\left\Vert v\right\Vert
_{L^{4}}\right)  \left\Vert \nabla\phi\right\Vert _{L^{2}}\left\Vert
w\right\Vert _{L^{4}}\\
& \leq\left(  \left\Vert u\right\Vert _{L^{2}}^{1/4}\left\Vert u\right\Vert
_{H^{1}}^{3/4}+\left\Vert v\right\Vert _{L^{2}}^{1/4}\left\Vert v\right\Vert
_{H^{1}}^{3/4}\right)  \left\Vert w\right\Vert _{L^{4}}\\
& \leq\left(  \left\Vert u\right\Vert _{L^{2}}^{1/2}\left\Vert u\right\Vert
_{H^{1}}^{3/4}+\left\Vert v\right\Vert _{L^{2}}^{1/2}\left\Vert v\right\Vert
_{H^{1}}^{3/4}\right)  \left\Vert w\right\Vert _{L^{2}}^{1/4}\left\Vert
w\right\Vert _{H^{1}}^{3/4}.
\end{array}
\label{12}%
\end{equation}
Which implies the following inequality%
\begin{equation}
\left\vert \left(  B\left(  v,w\right)  ,\phi\right)  -\left(  B\left(
w,u\right)  ,\phi\right)  \right\vert \leq\left(  \left\Vert u\right\Vert
_{H^{1}}^{3/4}+\left\Vert v\right\Vert _{H^{1}}^{3/4}\right)  \left\Vert
w\right\Vert _{H^{1}}^{3/4}. \label{13}%
\end{equation}
Then, we have%
\begin{equation}
\left\Vert B\left(  v,w\right)  -B\left(  w,u\right)  \right\Vert _{V^{\prime
}}^{2}\leq\left(  \left\Vert u\right\Vert _{H^{1}}^{3/2}\left\Vert
w\right\Vert _{H^{1}}^{3/2}+\left\Vert v\right\Vert _{H^{1}}^{3/2}\left\Vert
w\right\Vert _{H^{1}}^{3/2}\right)  . \label{14}%
\end{equation}
Integrating $(\ref{14})$ in time and using Holder's inequality with $p=q=2$,
we see that%
\begin{equation}%
\begin{array}
[c]{l}%
\int_{0}^{T}\left\Vert B\left(  v,w\right)  -B\left(  w,u\right)  \right\Vert
_{V^{\prime}}^{2}dt\leq\\
\multicolumn{1}{c}{\left(  \left(  \int_{0}^{T}\left\Vert u\right\Vert
_{H^{1}}^{3}dt\right)  ^{\frac{1}{2}}+\left(  \int_{0}^{T}\left\Vert
v\right\Vert _{H^{1}}^{3}dt\right)  ^{\frac{1}{2}}\right)  \left(  \int
_{0}^{T}\left\Vert w\right\Vert _{H^{1}}^{3}dt\right)  ^{\frac{1}{2}},}%
\end{array}
\label{15}%
\end{equation}
under the assumption that $u$, $v\in L^{3}\left(  0,T;V\right)  $ and since
$w\in L^{2}\left(  0,T;V\right)  $ it follows that $\partial_{t}w\in
L^{2}\left(  0,T;V^{\prime}\right)  $ and hence $w\in C^{0}\left(
0,T;H\right)  $, see \cite[Lemma III. 1. 2]{10}.
\end{proof}

We may summarize the preceding results as follows.

\begin{proposition}
Let $u$, $v$ two weak solutions of the 3D Navier-Stokes equations $(\ref{1})$
with the same initial datum $u\left(  0\right)  =v\left(  0\right)  $ and $u$,
$v\in L^{3}(0,T;V)$, then $w=u-v\in C\left(  0,T;H\right)  $ and
\begin{equation}
u=v\text{ for all }0\leq t\leq T. \label{16}%
\end{equation}

\end{proposition}

Let $\mathcal{I}^{\ast}=\left[  0,T^{\ast}\right]  $ be the maximal interval
contains $0$, on wich the weak solution given in Theorem 2.1 is finite in the
$H^{1}$-norm. We have $\mathcal{I}^{\ast}\subset C_{%
%TCIMACRO{\U{211d} }%
%BeginExpansion
\mathbb{R}
%EndExpansion
^{+}}\Sigma\ $and $T^{\ast}$ the maximal time of regularity in $\mathcal{I}%
^{\ast}$. Our results remains true in $\mathcal{I}^{\ast}$ and also if
$\Sigma$ is empty. In Serrin uniqueness result, for $p=5/2$ we get $q=15$ and
for $p=3$ we get $q=9$, which means an important improvement.

\end{document}